\documentclass[nofootinbib,11pt]{revtex4}

\usepackage{amsfonts}
\usepackage{amsmath}
\usepackage{amssymb}
\usepackage{amsthm}
\usepackage{mathrsfs}
\usepackage{bbm}
\usepackage{bm}
\usepackage[usenames,dvipsnames]{color}

\usepackage{hyperref}

\newcommand{\id}{\mathbbm{1}}
\newcommand{\ot}{\otimes}
\newcommand{\mc}{\mathcal}
\newcommand{\lbr}{\text{[}}
\newcommand{\rbr}{\text{]}}
\newcommand{\tr}{\triangleright}
\newcommand{\tl}{\triangleleft}

\newtheorem{prop}{Proposition}
\newtheorem{definition}{Definition}

\begin{document}

\title{Quantum kappa-deformed differential geometry and field theory}

\author{Flavio Mercati}
\email{flavio.mercati@gmail.com}
\affiliation{School of Mathematical Sciences, University of Nottingham, NG7 2RD, United Kingdom}

\begin{abstract}
I introduce in $\kappa$-Minkowski noncommutative spacetime the basic tools of quantum differential geometry, namely bicovariant differential calculus, Lie and inner derivatives, the integral, the Hodge-$\bm{*}$ and the metric. I show the relevance of these tools for field theory with an application
to complex scalar field, for which I am able to identify a vector-valued four-form
which generalizes the energy-momentum tensor. Its closedness is proved, expressing in a covariant form the conservation of energy-momentum.
\end{abstract}

\maketitle

\vspace{-36pt}

\tableofcontents

\section{Introduction}
\label{sec1}

The $\kappa$-deformation of the Poincar\'e algebra, $U(so(3,1)){\triangleright\!\!\!\blacktriangleleft} \mc A^*$ was introduced by Lukierski and
collaborators \cite{LukierskiInventsKpoincare1,LukierskiInventsKpoincare2}.
As the Poincar\'e algebra could not be deformed using the
Drinfeld--Jimbo scheme \cite{Drinfeld,Jimbo1,Jimbo2}, applicable only
to simple Cartan--Lie algebras, the starting point was the 
q-deformation of the 5-dimensional anti-de Sitter 
algebra $o(3,2) \to U_q(o(3,2))$.
Dimensionalizing the generator through the introduction
of a de Sitter radius $R$, one can perform the  the 
In\"on\"u--Wigner contraction of that algebra, $R \to \infty$ 
together with the $q \to 0$ limit, in
such a way that the dimensionful quantity
$$
\kappa^{-1} = R \, \log q~,
$$
remains constant. This contraction scheme was first introduced
by Celeghini and collaborators \cite{FirenzeGroup} and it is useful
to introduce a dimensionful deformation parameter.
Later Zakrzewski \cite{ZakrzewskiInventsKPGroup} used the same
r-matrix implied by the $\kappa$-Poincar\'e algebra to generate a
deformed Poisson structure on the Poincar\'e group, and its quantization
led to the $\kappa$-Poincar\'e quantum group $ \mc A{\triangleright\!\!\!\blacktriangleleft}\mathbbm{C}[SO(3,1)]$. This quantum group
was then proven to be dual to the $\kappa$-Poincar\'e algebra by
Kosinski and Maslanka. Majid and Ruegg clarified the bicrossproduct
structure of $\kappa$-Poincar\'e, consisting of a semidirect product of the classical Lorentz algebra $so(1,3)$ acting in a deformed way on the translation sector $\mc A^*$, and a backreaction of the momentum sector on the Lorentz 
transformations, which renders also the coalgebra semidirect.
This work allowed to introduce  in a consistent way an homogeneous
space of the $\kappa$-deformed symmetry $\mathcal A$, as the quotient Hopf algebra of the $\kappa$-Poincar\'e group with the Lorentz group
$ \mc A{\triangleright\!\!\!\blacktriangleleft}\mathbbm{C}[SO(3,1)]/\mathbbm{C}[SO(3,1)]$.
The result is a noncommutative Hopf algebra with primitive
coproduct, antipode and counit, which can be interpreted,
in a noncommutative-geometrical fashion, as the algebra of functions
over a noncommutative spacetime, called $\kappa$-Minkowski. Symbolically $\mathcal A \sim \mathbbm C_q[\mathbbm{R}^{3,1}]$.
Differential calculus is a fundamental tool which is necessary to study
field theory over $\kappa$-Minkowski. Several differential structures 
can be defined on a noncommutative  space, and the requirement
of bicovariance \cite{Woronowicz} is a particularly selective one.
Still, there are several bicovariant differential calculi, but in the
case of $\kappa$-Minkowski it makes sense to ask covariance under
the symmetries of this space, which are
encoded in the $\kappa$-Poincar\'e group. In particular, as 
Sitarz \cite{Sitarz} proved, there are no 4-dimensional bicovariant differential
calculi that are also Lorentz-covariant. The simplest calculus
that achieve this is 5-dimensional. This phenomenon of
the natural emergence of higher-dimensional calculi is a common
feature of several noncommutative spaces, as noticed by Majid \cite{MagidAlgebraicApproachII,majidbook}.

In this paper, I introduce a series of concepts which represent the
basis to do differential geometry in a noncommutative setting.
This allows to study field theory over $\kappa$-Minkowski, and
to construct manifestly Lorentz-covariant theories.
Particularly relevant, for this, have been the results of Radko and
Vladimirov \cite{Radko} and Brzezinski \cite{Brzezinski}, which are
extensively used in this paper. Several results are also based
on the star-product introduced by Sitarz and Durhuss \cite{SitarzDurhuss},
and the twisted graded trace introduced by the author with Sitarz \cite{KappaCorfu}.

In Section \ref{sec2}~I develop the 5-dimensional differential calculus
introduced by Sitarz beyond the one-forms, defining the entire
differential complex up to 5-forms, which are commutative and,
being isomorphic to 0-forms, close the complex.
Section \ref{sec3}~is devoted to the Lie derivative and the inner derivative,
exploiting the graded Hopf algebra construction for the differential
complex and its dual introduced in \cite{Radko}, from which a natural 
concept of Lie and inner derivative emerge.
In Section \ref{sec4}~I review the construction of the integral, that has
been defined in \cite{KappaCorfu}, and I derive some useful properties.
The last structure that I introduce is the Hodge-$*$. It is defined axiomatically
and it is then explicitly constructed. This is probably the most relevant
contribution of this paper, and is contained in Section \ref{sec5}.
Section \ref{sec6}~shows an application in field theory of all of the structures I introduced, the differential complex, the Lie and inner derivatives, and the Hodge-$*$. The application is the construction of a vector-valued 4-form, which is
the noncommutative analogous of the vector-valued three-form whose components are those of the energy-momentum tensor of a scalar field in Minkowski space. The closedness of this form express the energy and momentum conservation. This
reformulation in terms of differential forms of the conservation law allows
to express it in a manifestly covariant way, a feat that, without the language of differential forms is problematic in a noncommutative spacetime.
The last Section contains the conclusions.

\begin{center}{\bf Notation}\end{center}

Einstein's convention for the sum over repeated indices is assumed. Greex indices $\mu,\nu,\dots$ go from $0$ to $3$. Latin beginning-of-alphabet letters $a,b,\dots$ refer to indices going from $0$ to $4$. Latin letters following the $i$ ($j,k,l,\dots$) refer to spatial indices, $1,2,3$.  

Boldface symbols like $\bm \omega$ refer to $n$-forms, $n>0$, while regular ones (\emph{i.e.} $f$) refer to functions, or $0$-forms. 
With an overline ($\overline{z}$, $z \in \mathbbm{C}$) we indicate complex conjugation. The involution is represented with a dagger ($x^\dagger$, $x \in \Gamma^{\wedge}$ or $x \in U(so(3,1)){\triangleright\!\!\!\blacktriangleleft} \mc A^*$).

Symmetrization and antisymmetrization of indices are indicated by curly and square brackets:
$$
\omega_{\{123\}} = \omega_{123} + \omega_{213} + \omega_{231}  + \omega_{321} + \omega_{312} + \omega_{132} ~, \qquad
\rho_{[ab]} = \rho_{ab}  - \rho_{ba}~. 
$$

\section{Differential calculus over $\kappa$-Minkowski}
\label{sec2}

\subsection{$\kappa$-Minkowski and $\kappa$-Poincar\'e algebras}

The $\kappa$-Minkowski space was introduced by by Majid and Ruegg \cite{MajidRueggBicross}, as a homogeneous space of $\kappa$-deformed Poincar\'e symmetries. Majid and Ruegg identified the bicrossproduct structure of the 
$\kappa$-Poincar\'e algebra introduced by Lukierski, 
Nowicki and Ruegg \cite{LukierskiInventsKpoincare1,LukierskiInventsKpoincare2},
and this in turn allowed to correctly identify the homogeneous space
as a noncommutative space, dual to the translation subalgebra.

The $\kappa$-Minkowski algebra $\mc A$, understood
as an Hopf $*$-algebra (the involution is represented with the dagger $(\cdot)^\dagger$ operation) is generated by $x^\mu$, $\mu=0,\dots,3$,
\begin{equation}
\begin{array}{c}
[x_j , x_0] = \frac{i}{\kappa} \, x_j ~, ~~~Ê[x_j , x_k]= 0~, \vspace{6pt}\\
\Delta x^\mu = x^\mu \ot \id + \id \ot x^\mu ~,  \vspace{6pt}\\
\varepsilon(x^\mu) = 0 ~, \qquad S(x^\mu) = - x^\mu ~, \vspace{6pt}\\
(x^\mu)^\dagger = x^\mu,
\end{array}
\end{equation}
where $\kappa$ is a real deformation parameter. Commutative Minkowski spacetime
is obtained through the limit $\kappa \to 0$.
The translation algebra $\mc A^*$ is the dual Hopf $*$-algebra to $\kappa$-Minkowski,
\begin{equation}
\begin{array}{c}
[P_\mu, P_\nu]= 0~, \vspace{6pt}\\
\Delta P_0 = P_0 \ot \id + \id \ot P_0~, ~~~ \Delta P_j = P_j \ot \id + e^{-P_0/\kappa} \ot P_j~,  \vspace{6pt}\\
\varepsilon(P_\mu) = 0 ~, \qquad S(P_0) = - P_0 ~, \qquad S(P_j) = - e^{P_0/\kappa}  P_j ~, \vspace{6pt}\\
(P_\mu)^\dagger = P_\mu.
\end{array}
\end{equation}

The $\kappa$-Poincar\'e Hopf $*$-algebra $U(so(3,1)){\triangleright\!\!\!\blacktriangleleft} \mc A^*$ is the bicrossproduct generated by $N_j,R_k \in U(so(3,1))$, $P_\mu \in \mc A^*$, defined by the following additional relations
\begin{equation}
\begin{array}{c}
[N_j, P_k]= i \, \delta_{jk} \left(  \frac{\kappa}{2}(1 - e^{- 2 P_0/\kappa})  + \frac{1}{2\kappa}  |\vec{P}|^2 \right) - \frac{i}{\kappa} P_j P_k~, \vspace{6pt}\\
{[}N_j,P_0] = i \, P_j ~, ~~~ [R_j ,P_k ] = i \, \epsilon_{jkl} P_l  \, \vspace{6pt}\\
\Delta N_k  =  N_k \otimes \id + e^{-P_0/\kappa} \otimes N_k  + \frac{i}{\kappa} \epsilon_{klm}  P_l \otimes R_m~, ~~~ \Delta R_j = R_j \ot \id + \id \ot R_j ~, \vspace{6pt}\\
\varepsilon(N_j) = 0 ~, ~~~ \varepsilon(R_k) = 0Ê~, ~~~ S(N_j) = -e^{P_0/\kappa}N_j +\frac{i}{\kappa} \epsilon_{jkl} e^{\lambda P_0} P_k R_l ~, ~~~ S(R_k) = - R_k~, \vspace{6pt}\\
(N_j)^\dagger =N_j ~, ~~~Ê(R_k)^\dagger =R_k ~, .
\end{array} \label{BicrossRelations}
\end{equation}

The translation algebra acts covariantly from the left on $\kappa$-Minkowski,
\begin{equation}
t \tr x = x_{(1)}  \left<  t , x_{(2)} \right> ~, \qquad t \in \mc A^*~, f \in \mc A.
\end{equation}
and since the bicrossproduct construction involve a right action of $U(so(3,1))$ on $\mc A^*$, which is encoded into the commutators (\ref{BicrossRelations}):
\begin{equation}
t \tl h = [h,t]~, ~~ t \in \mc A^*~, h \in U(so(3,1))~,
\end{equation}
the $U(so(3,1))$ acts too from the left on $\mc A$, by dualizing the right-action on $\mc A^*$:
\begin{equation}
\left< t , h \tr x \right> = \left< t \tl h , x \right> ~, \qquad t \in \mc A^*~, x \in \mc A~, h \in U(so(3,1))~.
\end{equation}
Then there is a left covariant action of the whole $\kappa$-Poincar\'e algebra $U(so(3,1)){\triangleright\!\!\!\blacktriangleleft} \mc A^*$ on $\mc A$, which can be obtained by the action on the coordinate base $x^\mu$:
\begin{equation}
\begin{array}{c}
P_0 \tr x_0 = i ~, ~~ P_0 \tr x_j = 0Ê ~, ~~ P_j \tr x_0 = 0 ~, ~~  P_j \tr x_k = - i \, \delta_{jk}  ~,\vspace{6pt}\\
R_j \tr x_0 = 0 ~, ~~  R_j \tr x_k = \, \epsilon_{jkl} x_l ~, ~~ N_j \tr x_0 = x_j ~, ~~  N_j \tr x_k = \delta_{jk} x_0~,
\end{array}
\end{equation}
and extending it on products of coordinates through the coproducts of $U(so(3,1)){\triangleright\!\!\!\blacktriangleleft} \mc A^*$.

The left action of $\kappa$-Poincar\'e over $\kappa$-Minkowski is covariant under
involution, in the sense that
\begin{equation}
(h \tr x)^\dagger = S(h) \tr x^\dagger ~, \qquad h \in U(so(3,1)){\triangleright\!\!\!\blacktriangleleft} \mc A^*~, ~~ x \in \mc A~. \label{InvolutionAction}
\end{equation}

\subsection{Poincar\'e invariant differential calculus}

In \cite{Sitarz} a 5-dimensional bicovariant, Poincar\'e invariant differential calculus over
$\mc A$ is introduced. We refer to it as $\Gamma$. It is generated by $\bm{e}^\mu = \bm{d} x^\mu$ and $\bm{e}^4$, and it is an $\mc A$-$*$-bimodule\footnote{We can't make out of $\Gamma$ alone an Hopf algebra, because it cannot be closed under coproduct - the identity does not belong to $\Gamma$ (there is no such thing as the identity one-form).} defined by the following commutation relations
\begin{equation}
\begin{array}{c}
\lbr x_j , \bm{d} x_k \rbr = \frac{i}{\kappa} \delta_{jk}( \bm{d} t - \bm{e}^4) ~, \qquad  \lbr x_j , \bm{d} t \rbr = \frac{i}{\kappa} \bm{d} x_j~,  \vspace{6pt}\\
\lbr t , \bm{d} x_j \rbr = 0 ~, \qquad  \lbr t , \bm{d} t \rbr = \frac{i}{\kappa} \bm{e}^4~,  \vspace{6pt}\\
\lbr x_j , \bm{e}^4 \rbr = \frac{i}{\kappa} \bm{d} x_j ~, \qquad  \lbr t , \bm{e}^4 \rbr = \frac{i}{\kappa} \bm{d} t~,  \vspace{6pt}\\
(\bm{d} x^\mu)^\dagger = \bm{d} x^\mu~, \qquad (\bm{e}^4)^\dagger = \bm{e}^4 ~,
\end{array}\label{OneForms}
\end{equation}
which can be written in a more compact form as
\begin{equation}
[x^\mu , \bm{e}^\nu] =  \frac{i}{\kappa} (\eta^{\mu\nu} \bm{e}^0 - \eta^{0\nu} \bm{e}^\mu - \eta^{\mu\nu} \bm{e}^4 ) ~, \qquad [ x^\mu , \bm{e}^{4}] =   \frac{i}{\kappa} \bm{e}^\mu ~.
\label{OneFormsCompact}
\end{equation}
The rules above can be obtained from those calculated in \cite{Sitarz} with the substitutions
\begin{equation}
\begin{array}{l}
t \to - i \, t \\  x \to i \, x
\end{array}
~,~~
\begin{array}{l}
dt \to i \, \bm{d}t  \\ dx_j \to - i \, \bm{d}x_j 
\end{array}
~,~~
\phi \to  i \, \kappa \, \bm{e}^4~.
\end{equation}

The differential is a map $\bm d : \mc A \to \Gamma$ satisfying the Leibniz rule
\begin{equation}
\bm d (f g) = ( \bm d f ) g + f ( \bm d g ) ~,
\end{equation}
the commutation relations between functions and differential forms can be written as \cite{Radko}
\begin{equation}
\bm{e}^a f =   ( {\lambda^a}_b \tr f ) \bm{e}^b ~,
\end{equation}
where ${\lambda^a}_b \in \mc A^*$, and the differential map can be written as
\begin{equation}
\bm d f = (i \, \xi_a \tr f ) \, \bm{e}^a
\end{equation}
where, again, $\xi_a \in \mc A^*$. Then this, and the Leibniz rule for the differential imply that
the coproduct of $\xi_a$ is
\begin{equation}
\Delta( \xi_a ) =  \xi_b \otimes {\lambda^b}_a  + \id \otimes  \xi_a ~,
\end{equation}
and its antipode and counit are
\begin{equation}
S(\xi_a)  = - \xi_b \, S({\lambda^b}_a)  ~, \qquad \epsilon(\xi_a)=0 ~,
\end{equation}
and the associativity of the product between forms and functions ($\bm \omega (fg) = (\bm \omega f) g$, $\bm \omega \in \Gamma$ and $f,g, \in \mc A$) imply
\begin{equation}
\Delta(  {\lambda^a}_b  ) =   {\lambda^a}_c   \otimes  {\lambda^c}_b ~.
\end{equation}
From the formulas above the following additional properties can be derived \cite{Radko}:
\begin{equation}
\epsilon ( {\lambda^a}_b ) =  {\delta^a}_b ~, \qquad S(  {\lambda^a}_c )  {\lambda^c}_b =  {\delta^a}_b  ~.
\end{equation}

\begin{prop}
From the relation (\ref{OneForms}) one deduces the following expressions for ${\lambda^a}_b $,
\begin{equation}
{\lambda^a}_b = 
\left(\begin{array}{ccc}
\cosh \frac{P_0}{\kappa} + \frac{1}{2 \kappa^2} e^{P_0 / \kappa} |\vec P|^2 & ~~~~ \frac{1}{\kappa} \vec P ~~~~ & - \sinh \frac{P_0}{\kappa} - \frac{1}{2 \kappa^2} e^{P_0 / \kappa} |\vec P|^2
\\
\frac{1}{\kappa} e^{P_0/\kappa} \vec P & I & - \frac{1}{\kappa} e^{P_0/\kappa} \vec P
\\
- \sinh \frac{P_0}{\kappa} + \frac{1}{2 \kappa^2} e^{P_0 / \kappa} |\vec P|^2 &~~~~ \frac{1}{\kappa} \vec P ~~~~& \cosh \frac{P_0}{\kappa} - \frac{1}{2 \kappa^2} e^{P_0 / \kappa} |\vec P|^2
\end{array}\right) ~,
\end{equation}
where $I$ is the $3\times3$ identity matrix, and for $\xi_a$,
\begin{equation}
\xi_a = \left\{ - \kappa \, \sinh \frac{P_0}{\kappa} + \frac{1}{2 \kappa} e^{P_0 / \kappa} |\vec P|^2   ,   \vec P ,  \kappa \,  \cosh \frac{P_0}{\kappa} - \frac{1}{2 \kappa} e^{P_0 / \kappa} |\vec P|^2    - \kappa \right\} ~.
\end{equation}
\end{prop}
 
The elements ${\lambda^a}_b$ form a matrix which is an element of the 5-dimensional Lorentz group, $SO(4,1)$. In fact
\begin{equation}
\eta^{cd} {\lambda^a}_c {\lambda^b}_d = \eta^{ab} ~,
\end{equation}
where $\eta^{ab} = \text{diag} \{-1,1,1,1,1 \}$, as one can compute easily.

One could have adopted the opposite convention for the
commutation relations between functions and differential forms,
\begin{equation}
 f \, \bm{e}^a =    \bm{e}^b( {\sigma^a}_b \tr f ) ~,
\end{equation}
in this case, of course,
\begin{equation}
 {\sigma^a}_b = S({\lambda^a}_b) ~.
\end{equation}
Analogous relations hold for the differential
\begin{equation}
\bm d f = \bm{e}^a  ( i \, \chi_a \tr f ) ~, 
\end{equation}
where 
\begin{equation}
\chi_a = - S(\xi_a) ~, 
\end{equation}
the coproduct of $\chi_a$ is then
\begin{equation}
\Delta \chi_a = \chi_a \otimes \id + {\sigma^b}_a \otimes \chi_b ~.
\end{equation}

Now the elements $\chi_a$ are explicitly Lorentz covariant, in
the sense that, introducing a natural action of the $\kappa$-Poincar\'e
algebra $U(so(3,1)){\triangleright\!\!\!\blacktriangleleft} \mc A^*$
over one-forms in this way \cite{Sitarz},
\begin{equation}
h \tr (f \bm{d} g ) = (h^{(1)} \tr f) \bm{d} ( h^{(2)} \tr g ) ~, \qquad h \tr (\bm{d} f \,  g ) = \bm d (h^{(1)} \tr f) ( h^{(2)} \tr g ) ~,
\end{equation}
then 
\begin{equation}
R_j \tr \bm e_0 = 0 ~, ~~  R_j \tr \bm e_k = \, \epsilon_{jkl} \bm e_l ~, ~~ N_j \tr \bm e_0 = \bm e_j ~, ~~  N_j \tr \bm e_k = \delta_{jk} \bm e_0~,
\end{equation}
and the commutation relations between $\chi_a$ and $N_j$,$R_k$ are
\begin{equation}
[M_{\mu\nu},\chi_\rho] = i ( \eta_{\mu \rho} \chi_\nu - \eta_{\nu \rho} \chi_\mu) ~, \qquad
[M_{\mu\nu},\chi_4] = 0 ~,
\end{equation}
where $M_{0j} = N_j$ and $M_{jk} = \epsilon_{jkl} R_l$. Put in other way, $\chi_\mu$ transform like a 4-vector, while $\chi_4$ transforms like a scalar.

As a last remark for this section let's notice that both $\chi_a$ and $\xi_a$, when squared with the metric $\eta^{ab} = \text{diag} \{-1,1,1,1,1 \}$, generate the mass Casimir of $\kappa$-Poincar\'e \cite{LukierskiInventsKpoincare2},
\begin{equation}
\eta^{ab} \, \xi_a \, \xi_b = \eta^{ab} \, \chi_a \, \chi_b = \square_\kappa ~, \label{Casimir}
\end{equation}
which id invariant under antipode $S(\square_\kappa )  = \square_\kappa $ and is a central element of  $U(so(3,1)){\triangleright\!\!\!\blacktriangleleft} \mc A^*$.

\subsection{The differential complex (forms of degree higher than one)}

In \cite{Sitarz} it is shown that $\Gamma^2$ is generated by $\bm{e}^\mu \wedge \bm{e}^\nu  = - \bm{e}^\nu \wedge \bm{e}^\mu$ and $\bm{e}^\mu \wedge \bm{e}^4 = - \bm{e}^4 \wedge \bm{e}^ \mu $, with the additional relations
\begin{equation}
\bm{e}^j \wedge \bm{e}^j = - \bm{e}^0 \wedge \bm{e}^0 ~, 
\qquad
\bm{d e}^4 = i \, \kappa \, (\bm{e}^j \wedge \bm{e}^j - \bm{e}^0 \wedge \bm{e}^0 ) ~, 
\end{equation}
$\Gamma^2$  is another $\mc A$-$*$-bimodule, and the Jacobi identities applied to 
mixed products of the kind $x^\mu \bf{e}^a \wedge \bf{e}^b$ imply that\footnote{These relations were left as matters of choice in \cite{Sitarz}.}
\begin{equation}
\bm{e}^0 \wedge \bm{e}^0 = \bm{e}^1 \wedge \bm{e}^1=\bm{e}^2 \wedge \bm{e}^2=\bm{e}^3 \wedge \bm{e}^3=\bm{e}^4 \wedge \bm{e}^4 = 0 ~, \qquad \bm{d e}^4 = 0 ~;
\end{equation}
to make $\Gamma^2$ into an Hopf $*$-bimodule, we add the involution as
\begin{equation}
(\bm{e}^\mu \wedge \bm{e}^\nu)^\dagger = - \bm{e}^\mu \wedge \bm{e}^\nu ~, \qquad  (\bm{e}^\mu \wedge \bm{e}^4)^\dagger = - \bm{e}^\mu \wedge \bm{e}^4 ~.
\end{equation}


The commutation relations of all the $\Gamma^n$s can be found through the associative property
\begin{equation}
[ x^\mu , \bm{e}^{a_1}  \wedge \dots \wedge \bm{e}^{a_n} ] =
[ x^\mu , \bm{e}^{a_1}] \wedge \bm{e}^{a_2} \wedge \dots \wedge \bm{e}^{a_n}
+ \dots 
+ \bm{e}^{a_1} \wedge \bm{e}^{a_2} \wedge \dots \wedge [ x^\mu , \bm{e}^{a_n} ]~,
\end{equation}
and under involution the basic forms of $\Gamma^n$ behave as
\begin{equation}
\left(\bm e^{a_1} \wedge \dots \wedge \bm e^{a_n} \right)^\dagger = (-1)^{n(n-1)/2} \bm e^{a_1} \wedge \dots \wedge \bm e^{a_n} ~.
\end{equation}

Due to the graded-commutativity of the wedge product of base forms, $\Gamma^5$ is one-dimensional and is generated only by the (penta-) volume form $\bm{vol}^5 = \bm{e}^0 \wedge \bm{e}^1 \wedge \bm{e}^2 \wedge \bm{e}^3 \wedge \bm{e}^4$, which is self-adjoint
\begin{equation}
(\bm{vol}^5)^\dagger =  \bm{vol}^5 ~,
\end{equation}
and commutes with $\mc A$,
\begin{equation}
[x^\mu, \bm{vol}^5 ] = 0~,
\end{equation}
as can be easily proved by direct calculation.

In \cite{Brzezinski} is is shown that the entire exterior algebra $\Gamma^{\wedge} =  \mc{A} \oplus \Gamma \oplus \Gamma^2 \oplus \Gamma^3 \oplus \Gamma^4\oplus \Gamma^5$ can be made into a graded Hopf $*$-algebra, with coproduct
\begin{equation}
\Delta (\bm{e}^a) =  \bm{e}^a \otimes \id + \id \otimes \bm{e}^a~,
\end{equation}
antipode and counit
\begin{equation}
S(\bm{e}^a) = - \bm{e}^a ~, \qquad \epsilon(\bm{e}^a ) = 0 ~,
\end{equation}
where the extension of the multiplication to the tensor product is nontrivial, and satisfy the rule \cite{Brzezinski}
\begin{equation}
( \bm \omega \otimes \bm \rho ) \wedge (\bm \omega' \otimes \bm \rho' ) = (-1)^{nm} (\bm \omega \wedge \bm \omega' ) \otimes (\bm \rho \otimes \bm \rho' ) ~,
\end{equation}
where $\bm \omega, \bm \rho,\bm \omega',\bm \rho' \in \Gamma^{\wedge}$ and $\bm \rho$, $\bm \omega'$ are homogeneous forms of degree, respectively, $n$ and $m$.

Also the differential map can be extended to $\Gamma^{\wedge}$: it is a map $\bm d : \Gamma^n \to \Gamma^{n+1}$ obeying the graded Leibniz rule
\begin{equation}
\bm d ( \bm \omega \wedge \bm  \rho) = (\bm d \bm \omega) \wedge \bm \rho + (-1)^{n} \bm  \omega \wedge (\bm d  \bm \rho) ~,
\end{equation}
for $\bm \omega$ homogeneous  ($\bm \omega \in \Gamma^n$) and any $ \bm \rho \in \Gamma^{\wedge}$, and the nilpotency condition
\begin{equation}
\bm d \circ \bm d = 0 ~.
\end{equation}
The extension of $\bm d$ to the tensor product $\Gamma^{\wedge} \otimes \Gamma^{\wedge}$ is trivial \cite{Brzezinski}
\begin{equation}
\bm d ( \bm \omega \otimes \bm \rho) = (\bm d \bm \omega) \otimes \bm \rho + (-1)^{n} \bm \omega \otimes (\bm d \bm \rho) 
\end{equation}
for any $\rho \in \Gamma^{\wedge}$ and $\omega \in \Gamma^n$; the equation above implies that the coproduct of $\Gamma^{\wedge}$ and $\bm d$ commute:
\begin{equation}
\Delta \circ \bm d = \bm d \circ \Delta ~.
\end{equation}

From the covariance of the action of $\mc A^*$ under involution (\ref{InvolutionAction}), we deduce the covariance of the differential
\begin{equation}
\bm d ( \bm \omega^\dagger) = (-1)^n \, (\bm d  \bm \omega)^\dagger ~.
\end{equation}

\section{The $\kappa$-deformed Lie and inner derivatives}
\label{sec3}

In \cite{Radko} a (graded) Hopf algebra is built from $\Gamma^{\wedge}$ and $(\Gamma^{\wedge})^* = \mc A^* \oplus \Gamma^* \oplus (\Gamma^2)^*  \oplus (\Gamma^3)^*  \oplus (\Gamma^4)^*  \oplus (\Gamma^5)^*$ as the cross product $\Gamma^{\wedge} \rtimes (\Gamma^{\wedge})^*$.
The duality relations between  $\Gamma^{\wedge}$ and $(\Gamma^{\wedge})^*$ are such that
\begin{equation}
\left< \xi , \bm \omega \right> = 0 ~Ê\Leftarrow ~ \xi \in (\Gamma^n)^* , \bm \omega \notin \Gamma^n ~,
\end{equation}
and $\left< \xi , \omega \right>$ reduces to the duality relation between $(\Gamma^n)^*$ and  $\Gamma^n$ when $\xi \in (\Gamma^n)^*$ and $\omega \in \Gamma^n $.

\subsection{Lie derivative}

The algebra $\mc A^*$ in \cite{Radko} is interpreted as the space of left-invariant vector fields on $\mc A$, and the Lie derivative along an element $h$ of $\mc A^*$ is defined as the adjoint action of $\mc A^*$ over $\Gamma^{\wedge} \rtimes (\Gamma^{\wedge})^* $:
\begin{equation}
\bm \pounds_h := h \tr_{\text{ad}}~,
\end{equation}
that, on $\Gamma^{\wedge}$, reduces to
\begin{equation}
\bm \pounds_h \tr \bm \omega = \bm \omega_{(1)} \left< h , \bm \omega_{(2)} \right > ~, ~~ \forall \, \bm \omega \in \Gamma^{\wedge}~.
\end{equation}

Then the Lie derivative of forms along the vector field $h \in \mc A^*$ can be defined as a map $\bm \pounds_h : \Gamma^n \rightarrow \Gamma^n$ such that
\begin{equation}
\bm \pounds_h (\bm \omega) = ( h \tr  \omega_{a_1 \dots a_n} )\, \bm{e}^{a_1} \wedge \dots \wedge \bm{e}^{a_n} ~,
\end{equation}
for all $\bm \omega =  \omega_{a_1 \dots a_n}\, \bm{e}^{a_1} \wedge \dots \wedge \bm{e}^{a_n} \in \Gamma^n$.
The Lie derivative of products of forms depend on the coproduct of the vector field $h$:
\begin{equation}
\bm \pounds_h  ( \bm \omega \wedge \bm \rho ) = \bm \pounds_{h^{(1)}}  ( \bm \omega) \wedge \bm \pounds_{h^{(2)}} (\bm \rho ) ~,
\end{equation}
the coproduct of $h$ is in general non-primitive, with the exception of $P_0$ (the dual element to $x^0$), so in general $\bm \pounds_h$ does not satisfy the (graded) Leibniz rule.
We conclude this subsection with the observation, reported in \cite{Radko}, that the Lie derivative commutes with the differential,
\begin{equation}
\bm \pounds_h  \circ \bm d = \bm d \circ \bm \pounds_h ~.
\end{equation}

\subsection{Inner derivative}

The authors of \cite{Radko} propose to relate inner derivations with elements of $\Gamma^*$.
Starting from the base $\{ \theta_0,\theta_1,\theta_2,\theta_3,\theta_4\}$ of $\Gamma^*$ which is dual to the base $\{ \bm e^0,\bm e^1,\bm e^2,\bm e^3,\bm e^4\}$ of $\Gamma$ we define\footnote{The definition of the dual base in \cite{Radko} was different:
$\left< \theta_a ,  f \, \bm e^b \right> = \epsilon(f) \, {\delta^b}_a$. Here we need to put $f$ 
on the right to enforce Lorentz covariance (see below).}:
\begin{equation}
\left< \theta_a ,  \bm e^b  \, f \right> = \epsilon(f) \, {\delta^b}_a ~, \qquad f \in \mc A~, \label{DualityBracketsInnerDerivative}
\end{equation}
and has zero bracket with the other elements of $\Gamma^{\wedge}$
\begin{equation}
\left< \theta_a , \bm \omega \right> = 0 ~, ~~ \bm \omega \in \mc A,\Gamma^2,\Gamma^3,\Gamma^4,\Gamma^5~, 
\end{equation}
one defines a base of inner derivations $\bm i_a : \Gamma^n \to \Gamma^{n-1}$ in this way:
\begin{equation}
\bm i_a := \theta_a \tr_{\text{ad}} ~. 
\end{equation}
The inner derivation of functions is then zero,
\begin{equation}
\bm i_a (f) = f^{(1)} \left<\theta_a , f^{(2)} \right> = 0 ~, \qquad \forall \, f \in \mc A~,
\end{equation}
and that of one-forms is
\begin{equation}
\bm i_a (\bm \omega ) =  \omega_a~, \qquad \forall \, \bm \omega = \bm e^a \,  \omega_a\in \Gamma ~,
\end{equation}
while that of two-forms is
\begin{equation}
\bm i_a (\bm \omega ) =  \bm e^b \, (\omega_{ab} - \omega_{ba}) ~, \qquad \forall \, \bm \omega =  \bm e^a \wedge \bm e^b \, \omega_{ab}  \in \Gamma^2 ~.
\end{equation}
In general the inner derivative of an $n$-form can be written as
\begin{equation}
\bm i_a \omega = {\delta^{[b_n}}_a \bm{e}^{b_1} \wedge \dots \wedge \bm{e}^{b_{n-1}]}   \, \omega_{b_1 \dots b_n}~, \qquad \bm \omega =  \bm{e}^{a_1} \wedge \dots \wedge \bm{e}^{a_n} \,  \omega_{a_1 \dots a_n} \in \Gamma^n~.
\end{equation}

The inner derivative does not satisfy the graded Leibniz rule like in the commutative case: in fact, for example, the wedge product of two 2-forms is
\begin{equation}
\bm \omega \wedge \bm \rho = \bm e^a \, \omega_a \wedge \bm e^b \, \rho_b =
\bm e^a  \wedge \bm e^b  \, ( {\sigma^c}_b \tr \omega_a)  \rho_c~É 
\end{equation}
so that the inner derivative of $\bm \omega \wedge \bm \rho$ is
\begin{equation}
\bm i_a (\bm \omega \wedge \bm \rho) =   \left[  ( {\sigma^c}_b \tr \omega_a)  \rho_c -( {\sigma^c}_a \tr \omega_b)  \rho_c \right]\, \bm e^b ~.
\end{equation}

In \cite{Radko} it is shown that the Cartan identity for the Lie, inner and exterior derivatives holds without moifications,
\begin{equation}
\bm \pounds_{\chi_a}  = \bm d \circ \bm i_a + \bm i_a \circ \bm d~,
\end{equation}
and we see that our choice for the duality brackets defining the inner derivative (\ref{DualityBracketsInnerDerivative}) selects $\bm \pounds_{\chi_a} $, which is Lorentz-covariant in the sense that
\begin{equation}
\bm \pounds_{\chi_\mu}  ( M_{\rho\sigma} \tr f ) = M_{\rho\sigma} \tr \bm \pounds_{\chi_\mu} (f) - i \, \eta_{\mu [\rho}  \bm \pounds_{\chi_{\sigma]}}(f)  ~, \qquad f \in \mc A~.
\end{equation}

\section{Twisted-cyclic integral}
\label{sec4}

This section summarizes the results contained in \cite{KappaCorfu} about the
twisted graded trace. In addition, we remark the compatibility of the trace
with the involution, and we extend the results to the $3+1$-dimensional case.

Exploiting the commutativity between $\Gamma^5$ and $\mc A$ we introduce a left-invariant integral as in \cite{KappaCorfu}. The integral is a linear map
\begin{equation}
\int : \Gamma^5 \to \mathbbm{C}~,
\end{equation}
which respects the involution\footnote{This property is a straightforward consequence of the definition given in \cite{KappaCorfu}. The integral is introduced there as the standard Lebesgue integral over $\mathbb R^2$ (generalization to 4 dimensions is straightforward), applied to the functions which give a realization of the $\kappa$-Poincar\'e algebra through the $*$-product introduced in \cite{SitarzDurhuss}. The involution too has a realization in terms of these functions, and the property (\ref{InvolutionIntegral}) can be deduced from it.}
\begin{equation}
\overline{\left(\int f\right)} = \int f^\dagger~, \label{InvolutionIntegral}
\end{equation}
and is invariant under the left action of the whole $\kappa$-Poincar\'e algebra $U(so(3,1)){\triangleright\!\!\!\blacktriangleleft} \mc A^*$
\begin{equation}
\int h \triangleright \bm \rho = \varepsilon(h) \int \bm \rho~, \qquad \forall \, h \, \in \, U(so(3,1)){\triangleright\!\!\!\blacktriangleleft} \mc A^*
\end{equation}
where the action of $\kappa$-Poincar\'e algebra $U(so(3,1)){\triangleright\!\!\!\blacktriangleleft} \mc A^*$ on 5-forms is trivially induced from the action on $\mc A$:
\begin{equation}
\bm \rho = \rho \, \bm{vol}^5  ~~Ê\Rightarrow ~~Êh \triangleright \bm \rho ~ \dot = ~ ( h \triangleright \rho) \bm{vol}^5  ~.
\end{equation}

\begin{prop}
The integral is \emph{closed}, in the sense that
\begin{equation}
\int \bm d \bm \omega = 0~, \qquad \forall \, \bm \omega \, \in \, \Gamma^4~,
\end{equation}
\end{prop}

\begin{proof}
The closedness follows from the left-invariance under the action of $\mc A^*$, in fact for every 4-form $\bm  \omega$:
\begin{equation}
\bm d \bm \omega  = (\xi_a \tr \omega_{b_1 b_2 b_3 b_4} ) \, \bm e^a \wedge  \bm e^{b_1} \wedge \dots \wedge  \bm e^{b_2} = \xi_{[0} \tr \omega_{1234]} \, \bm{vol}^5 ~,
\end{equation}
then 
\begin{equation}
\int \bm d \bm \omega  = \varepsilon( \xi_{[0})  \int \omega_{1234]} \, \bm{vol}^5  = 0~.
\end{equation}
\end{proof}

The integral, however, is not cyclic with respect to the product of 5-forms with elements of $\mc A$. Instead, it satisfies a twisted cyclic property \cite{Kustermans,KappaCorfu}:
\begin{equation}
\int f g \,\bm{vol}^5 = \int g (T \tr f ) \, \bm{vol}^5 ~,
\end{equation}
where $T \in \mc A^*$ is an automorphism of  $\mc A$:	
\begin{equation}
\Delta(T) = T \otimes T ~, \qquad S(T) T = T S(T) = \idÊ~, \qquad \epsilon(T) = 1~.
\end{equation}
The explicit expression of $T$ is
\begin{equation}
T = e^{3 P_0 /\kappa} ~,
\end{equation}
which is the same result obtained in \cite{KappaCorfu}, but elevated to the third power.

\begin{prop}
The twisted cyclicity property holds also for products of forms.
\end{prop}
\begin{proof}
If  $ \bm \omega \, \in \, \Gamma^n~, ~ \bm \rho \, \in \, \Gamma^{5-n}$
\begin{eqnarray}
\int \bm \omega \wedge \bm \rho
&=& \int   \omega_{a_1 \dots a_n} \,  \bm e^{a_1} \wedge \dots  \wedge \bm e^{a_n} \wedge  \rho_{a_{n+1}\dots a_5}  \bm e^{a_{n+1}} \wedge \dots  \wedge \bm e^{a_5}
\nonumber\\
&=&  \int   \omega_{a_1 \dots a_n} \left( {\lambda^{a_1}}_{b_1} \dots {\lambda^{a_n}}_{b_n} \tr \rho_{a_{n+1}\dots a_5}\right)  \bm e^{b_1} \wedge \dots  \wedge \bm e^{b_n}   \wedge \bm e^{a_{n+1}} \wedge \dots  \wedge \bm e^{a_5}
\nonumber\\
&=&  \frac{1}{5!} \epsilon^{b_1 \dots b_n a_{n+1} \dots a_5} \int   \omega_{a_1 \dots a_n} \left( {\lambda^{a_1}}_{b_1} \dots {\lambda^{a_n}}_{b_n} \tr \rho_{a_{n+1}\dots a_5}\right) ~,
\nonumber 
\end{eqnarray}
now the ${\lambda^a}_b$s are $SO(4,1)$ matrices, then they leave the 5-dimensional Levi-Civita symbol invariant, and the following is true:
\begin{equation}
 {\lambda^{a_1}}_{b_1} \dots {\lambda^{a_n}}_{b_n} \epsilon^{b_1 \dots b_n a_{n+1} \dots a_5} =   S({\lambda^{a_{n+1}}}_{b_{n+1}}) \dots S({\lambda^{a_5}}_{b_{5}})  \epsilon^{a_1 \dots a_n b_{n+1} \dots b_5} ~, 
 \end{equation}
 then
 \begin{eqnarray}
&&\int \bm \omega \wedge \bm \rho
=  \frac{1}{5!} \epsilon^{a_1 \dots a_n b_{n+1} \dots b_5}  \int   \omega_{a_1 \dots a_n} \left( S({\lambda^{a_{n+1}}}_{b_{n+1}}) \dots  S({\lambda^{a_5}}_{b_{5}})  \tr \rho_{a_{n+1}\dots a_5}\right) ~,
\nonumber 
\\
&&=  \frac{1}{5!} \epsilon^{a_1 \dots a_n b_{n+1} \dots b_5}  \int  \left( S({\lambda^{a_{n+1}}}_{b_{n+1}}) \dots  S({\lambda^{a_5}}_{b_{5}})  \tr \rho_{a_{n+1}\dots a_5}\right) \left(T \tr\omega_{a_1 \dots a_n} \right) ~,
\nonumber 
\\
&&=   \int  \, \rho_{a_{n+1}\dots a_5}  \left(  {\lambda^{a_{n+1}}}_{b_{n+1}} \dots {\lambda^{a_5}}_{b_{5}}  T \tr\omega_{a_1 \dots a_n} \right) \bm e^{a_1} \wedge \dots \wedge \bm e^{a_n} \wedge \bm e^{ b_{n+1}}\wedge \dots \wedge \bm e^{b_5}  ~,
\nonumber 
\\
&&=   \int  \, \rho_{a_{n+1}\dots a_5} \, \bm e^{ a_{n+1}}\wedge \dots \wedge \bm e^{a_5} \left( T \tr\omega_{a_1 \dots a_n} \right) \bm e^{a_1} \wedge \dots \wedge \bm e^{a_n}  ~,
\nonumber 
\\
&&= \int \bm \rho \wedge ( T \tr \bm \omega ) ~. \nonumber
\end{eqnarray}
\end{proof}

%

\section{Hodge-$*$ and metric}
\label{sec5}

We now introduce for the first time a metric structure in $\kappa$-Minkowski, through the Hodge-$\bm *$.

\begin{definition}
The Hodge-$\bm *$ is an involutive map 
\begin{equation}
\bm{*} : \Gamma^n \to \Gamma^{5-n}~, \qquad \bm{*} \circ \bm{*} = (-1)^{n(5-n)} id ~,
\end{equation}
which is left and right $\mc A$-linear:
\begin{equation}
\bm{*}( f \, \bm \omega ) = f \, \bm{*}( \bm \omega ) ~, \qquad  \bm{*}( \bm \omega \, f ) = \bm{*}( \bm \omega ) \, f~,
\end{equation}
such that the following sesquilinear form (symplectic form)
\begin{equation}
( \bm \omega , \bm \rho ) = \int \, \bm \omega^\dagger \wedge \bm{*} \bm \rho ~, \qquad  \bm \omega , \bm \rho  \in \Gamma^n~,
\end{equation}
is a nondegenerate (indefinite) inner product between forms of the same degree, that is,
\begin{equation}
\overline{( \bm \omega , \bm \rho )} = ( \bm \rho , \bm \omega )~.
\end{equation}
\end{definition}

\begin{prop}
The Hodge-$\bm{*}$ is defined by the following rules
\begin{equation}
\left\{\begin{array}{l}
\bm{*}(\id) = \bm{vol}^5
\\
\bm{*}(\bm e^a) =  \frac{1}{4!} \eta^{ab} \epsilon_{bcdef} \, \bm e^c \wedge \bm e^d \wedge \bm e^e \wedge \bm e^f ~,
\\
\bm{*}(\bm e^a \wedge \bm e^b) = \frac{1}{3!} \eta^{ac}\eta^{bd} \epsilon_{cdefg} \, \bm e^e \wedge \bm e^f \wedge \bm e^g ~,
\end{array}\right. \label{Hodge*Rules}
\end{equation}
where $\eta^{ab} = \text{diag} \{-1,1,1,1,1 \}$,
and $\varepsilon_{abcde}$ is the 5-dimensional Levi-Civita symbol.
\end{prop}

\begin{proof}
For the left and right $\mc A$-linearity It is sufficient to prove the compatibility of the commutation rules (\ref{OneFormsCompact}), which means
\begin{equation}
[x^\mu , \bm{*}(\bm{e}^\nu)] =  \frac{i}{\kappa} (\eta^{\mu\nu} \bm{*}(\bm{e}^0) - \eta^{0\nu} \bm{*}(\bm{e}^\mu ) - \eta^{\mu\nu} \bm{*}(\bm{e}^4) ) ~, \qquad [ x^\mu , \bm{*}(\bm{e}^{4})] =   \frac{i}{\kappa}\bm{*}(\bm{e}^\mu) ~,
\end{equation}
this can be verified by direct calculation.

It remains to show that the inner product is well-defined. For 0-forms it trivially descends from the compatibility of the integral with the involution. For n-forms we have
\begin{eqnarray}
\int \bm \omega^\dagger \wedge \bm{*} \bm \rho
&=&  \frac{1}{(5-n)!} \eta^{c_1 b_1}\dots\eta^{c_n b_n} \epsilon_{c_1 \dots c_5}
\int    \bm e^{a_1} \wedge \dots  \wedge \bm e^{a_n} \left( \omega_{a_1 \dots a_n}^\dagger\,  \rho_{b_1\dots b_n} \right) \bm e^{c_{n+1}} \wedge \dots  \wedge \bm e^{c_5}
\nonumber\\
&=&  \frac{1}{(5-n)! 5!} \eta^{c_1 b_1}\dots\eta^{c_n b_n} \epsilon_{c_1 \dots c_5} \epsilon^{a_1 \dots a_n c_{n+1}\dots c_5}
\int  \, \omega_{a_1 \dots a_n}^\dagger\,  \rho_{b_1\dots b_n}
\\
&=&  \frac{1}{(5-n)!} \eta^{a_1 b_1}\dots\eta^{a_n b_n} \int  \, \omega_{a_1 \dots a_n}^\dagger\,  \rho_{b_1\dots b_n}  \nonumber
\\
&=& \frac{1}{(5-n)!} \eta^{a_1 b_1}\dots\eta^{a_n b_n} \overline{\left( \int \, \rho_{b_1\dots b_n}^\dagger \, \omega_{a_1 \dots a_n} \right)} = \overline{ \left( \int \bm \rho^\dagger \wedge \bm{*} \bm \omega \right) } ~.
\nonumber
\end{eqnarray}

The nondegeneracy
\begin{equation}
( \bm \omega , \, \cdot \,  ) = 0 ~~~ \Leftrightarrow ~~~ \bm \omega = 0 ~.
\end{equation}
can be proven by considering the following object:
\begin{equation}
( \bm \omega ,  \tilde {\bm \omega}  ) = \sum_{a,b,\dots} \int \bm (\omega^{ab\dots})^\dagger \omega^{ab\dots}
\end{equation}
where $ \tilde { \omega}^{ab\dots}  = -  \omega^{ab\dots} $ if the number of indices $ab\dots$ which are zero is odd, and $ \tilde { \omega}^{ab\dots}  =  \omega^{ab\dots} $ otherwise.
Then $( \bm \omega ,  \tilde {\bm \omega}  )$ is a positive sum of terms of the type
$\int f^\dagger f$, and the following chain of implications follow:
$( \bm \omega , \bm \rho  ) = 0 ~~~ \forall \bm \rho$ ~~ $\Rightarrow$ ~~ $( \bm \omega ,  \tilde {\bm \omega}  ) = 0 $ ~~ $\Rightarrow$ ~~ $\int \bm (\omega^{ab\dots})^\dagger \omega^{ab\dots} = 0~~~ \forall a,b,\dots $.

We have then reduced the nondegeneracy of the inner product to the nondegeneracy of the norm of 0-forms.
\begin{equation}
\int f^\dagger f = 0 ~~ \Leftrightarrow ~~ f = 0 ~,
\end{equation}
I will prove this in the 2-dimensional case (3-dimensional differential calculus),
to exploit the results of \cite{KappaCorfu}. Generalization to 4 dimensions pose
no difficulties. From Eq. (2.2) and (2.3) of \cite{KappaCorfu} we can represent the $f^\dagger f$ through the star product as
\begin{equation}
(f^\dagger f)(\alpha,\beta) = \frac{1}{(2\pi)^2} \int dv du \,\int dw dz \, \bar
f(\alpha+u  + w , e^{-z/\kappa}\beta) \,f(\alpha,e^{-v/\kappa} \beta ) \, e^{-i (uv+wz)}~,
\end{equation}
the trace is represented as the ordinary Lebesgue integral over $\mathbb{R}^2$, so upon a simple change of variable we get
\begin{equation}
\int (f^\dagger f)(\alpha,\beta) = \frac{1}{(2\pi)^2} \int d \alpha d \beta \, \bar F(\alpha,\beta) \, F(\alpha,\beta)~,
\end{equation}
where $F(\alpha,\beta) = \frac{1}{2\pi} \int dv du \,
e^{i uv} \,f(\alpha + u,e^{-v/\kappa} \beta ) = f^\dagger (-\alpha,\beta)~,$
so that $\int (f^\dagger f)(\alpha,\beta)  = 0 ~ \Leftrightarrow ~ F(\alpha,\beta) = 0 $.
The proof is concluded by the observation that $F(\alpha,\beta) = 0 ~Ê \Rightarrow ~ f = 0$, which is trivial.

\end{proof}

The Hodge-$\bm{*}$ defined in this way is covariant under the action of the $\kappa$-Poincar\'e algebra $U(so(3,1)){\triangleright\!\!\!\blacktriangleleft} \mc A^*$, in the sense that
\begin{equation}
\bm{*} (h \tr \, \bm \omega ) = h \tr \, \bm{*}( \bm \omega ) ~, \qquad h \in U(so(3,1)){\triangleright\!\!\!\blacktriangleleft} \mc A^*~, \label{HodgeCovariance}
\end{equation}
in fact covariance under translations is trivial, because they have null action on basic forms
$$
P_\mu \tr \bm e^{a_1} \wedge \dots \wedge \bm e^{a_n} = 0~,
$$
and the left and right $\mc A$-linearity of the Hodge-$\bm *$ implies eq. (\ref{HodgeCovariance}) for $h \in \mc A^*$.

Lorentz covariance is no less straightforward. The action of both boost $N_j$ and rotation $R_k$ generators on basic forms is primitive, in the sense that on products of $\bm e^a$ they act with the Leibniz rule, \emph{e.g.}:
$$
N_j \tr (\bm e^a \wedge \bm e^b) = (N^{(1)}_j \tr \bm e^a ) \wedge (N^{(2)}_j \tr \bm e^a ) = $$
$$
=(N_j \tr \bm e^a ) \wedge \bm e^b + ( e^{-\lambda P_0} \tr \bm e^a ) \wedge (N_j \tr \bm e^b) + \epsilon_{jkl} (P_k \tr \bm e^a) \wedge (R_l \tr \bm e^b )
$$
$$
=(N_j \tr \bm e^a ) \wedge \bm e^b + \bm e^a  \wedge (N_j \tr \bm e^b) ~,
$$
and they both have classical action over a single one-form \cite{Sitarz}:
\begin{equation}
N_j \, \tr \, \bm e^k = i \, {\delta^k}_j \, \bm e^0 ~, \qquad N_j \, \tr \, \bm e^0 = - i \, \bm e^j ~, \qquad N_j \, \tr \, \bm e^4 = 0 ~,
\end{equation}
\begin{equation}
R_j \, \tr \, \bm e^k = i \, \epsilon_{jkl} \, \bm e^l ~, \qquad R_j \, \tr \, \bm e^0 = 0 ~, \qquad R_j \, \tr \, \bm e^4 = 0 ~,
\end{equation}
so it's easy to see that the rules (\ref{Hodge*Rules}) are covariant. Then the covariance for general forms is proven through the left and right $\mc A$-linearity.

The $\kappa$-Hodge-$*$ induces a \emph{metric}, understood as a sesquilinear map of one-forms $g: \Gamma \otimes \Gamma \to \mathcal A$
\begin{equation}
g( \bm \omega , \bm \rho ) = * ( \bm \omega^\dagger \wedge * \bm \rho )Ê~,
\end{equation}
which is hermitian
\begin{equation}
g( \bm \omega , \bm \rho ) =g( \bm \rho, \bm \omega )^\dagger~.
\end{equation}
If applied to the basis forms the metric gives its components
\begin{equation}
g ( \bm e^a , \bm e^b ) = \eta^{ab} ~.
\end{equation}

\newpage
\section{Classical field theory}
\label{sec6}

With $\kappa$-deformed \emph{classical} field theory I mean any theory
which substitutes elements of $\mc A$ or $\Gamma^{\wedge}$ to scalar or
tensor fields, and which is based on a variational principle or simply on
equations of motion, which identify some subset of $\mc A$ (or $\Gamma^{\wedge}$) as the space of solutions. A $\kappa$-deformed \emph{quantum} field
theory should be based on an appropriately defined measure over $\mc A$,
and an associated partition function, allowing to perform a path integral.
The understanding of classical field theory should prelude the study of
quantum field theory over $\kappa$-Minkowski, as is the case also in
the commutative Minkowski space.

There is a very strong physical motivation for the study of $\kappa$-deformed
field theory, coming from 2+1-dimensional baground-independent quantum gravity \cite{FreidelLivine}.
In 2+1 dimensions, Einstein gravity reduces to a topological field theory
which is solvable, and quantizable with a path integral through spin-foam techniques.
Coupling this theory to a scalar field, and integrating out the gravitational
degrees of freedom (which corresponds to taking the ``no-gravity'' limit
$G\to0$) one ends up with an effective partition function for the scalar field,
in which the field is valued in $\mc A$, the 2+1 dimensional $\kappa$-Minkowski space. This gives an indication that $\kappa$-Minkowski may be the fundamental
state of quantum gravity, and its noncommutativity could be a manifestation
of the non-local correlations induced on fields by the quantum gravitational degrees
of freedom. This calculation cannot be performed in 3+1 dimension
where quantum gravity is not understood, but is nevertheless one of the 
most significant results in quantum gravity, pointing out that its fundamental
state is likely not to be a classical spacetime.

Therefore the study of field theory over $\kappa$-Minkowski is very relevant
for physics, as it may provide the interface between quantum gravity,
noncommutative geometry and their observable manifestations.
Today there is a fairly large literature on $\kappa$-deformed field theory 
\cite{LukierskiFieldTheory,gacmajid,kappanoether,Nopure,Antonino5D,KowalskiDeSitter,KowalskiMicheleKappa}. However, until now, it was hard to build a field theory
which is manifestly Lorentz covariant, as the only tool at disposal to define
``vectors'' was Sitarz' differential calculus, and one needs much more: at
least the higher-degree forms, but also an Hodge-$*$ to create maximal
degree forms out of them, and an integral to form an action. I introduced
all of these structure with the explicit  purpose of making this possible,
as I'll show in this section.

Some authors have considered the problem
of establishing an analogue of the Noether theorem in these theories,
associating conserved charges to  the symmetries of $\kappa$-Minkowski
\cite{kappanoether,Nopure,Antonino5D,KowalskiSchifo}.
However, the Lorentz covariance of the conserved charges that were
found was never considered, and the meaning of such conservation 
laws remained obscure.
I present a geometrical way of understanding the conservation laws,
which is allowed by the differential-geometrical tools I developed in 
the previous sections. A conservation law is expressed as the closure
of a current vector-valued 4-form. This form is the energy-momentum tensor
expressed in the language of differential forms. In the commutative
Minkowski space its analogue is a vector-valued 3-form, but here
we need a 4-form due to the additional dimension of the differential
calculus. To calculate this current 4-form, I will need also the Lie and
inner derivative, and all of the new structure I introduced in this paper
will then find an application in field theory.

\subsection{Scalar field}

As action for a complex scalar field we take
\begin{equation}
S = \frac{1}{2} \int \left\{ (\bm d \phi)^\dagger \wedge \bm{*}(\bm d \phi)
+ m^2 \phi^\dagger \wedge \bm{*}(\phi) \right\}~,
\end{equation}

with some calculations we can see that this action is the same as that
used in \cite{Nopure}
\begin{eqnarray}
S &=& \frac{1}{2} \int \left\{ (\xi_a \tr \phi^\dagger) \bm e^a \wedge \xi_b \tr  \phi \bm{*}(\bm e^b)
+ m^2 \phi^\dagger \, \phi \, \bm{vol}^5 \right\}~,
\nonumber \\ &=&
 \frac{1}{2} \int \left\{ \eta^{bc} (\xi_a \tr \phi^\dagger) ({\lambda^a}_b \xi_c \tr  \phi )
+ m^2 \phi^\dagger \, \phi \right\} \bm{vol}^5 ~,
\nonumber \\ &=&
 \frac{1}{2} \int \left\{ - \phi^\dagger (\eta^{bc}  \xi_a  \xi_b \tr  \phi )
+ m^2 \phi^\dagger \, \phi \right\} \bm{vol}^5 ~,
\nonumber \\ &=&
 \frac{1}{2} \int \left\{ - \phi^\dagger (\square_\kappa \tr  \phi )
+ m^2 \phi^\dagger \, \phi \right\} \bm{vol}^5 ~.
\end{eqnarray}

To calculate the equations of motion we make use of a variational procedure,
which, written in Fourier transform following the techniques shown in \cite{SitarzDurhuss,KappaCorfu} gives the same results and makes perfectly sense
\begin{eqnarray}
\delta S &=& \frac{1}{2} \int \left\{ \bm d (\delta \phi^\dagger) \wedge \bm{*}(\bm d \phi) +\bm d (\phi^\dagger) \wedge \bm{*}(\bm d \delta \phi)
+ m^2 \delta \phi^\dagger \wedge \bm{*}(\phi) + m^2 \phi^\dagger \wedge \bm{*}(\delta \phi) \right\}
\nonumber \\ &=&
 \frac{1}{2} \int \left\{\delta \phi^\dagger \left[  m^2 \, \bm{*}(\phi)  - \bm d \, \bm{*} \, \bm d \, (\phi) \right] +\bm d (\phi^\dagger) \wedge \bm{*}(\bm d \delta \phi)
+ m^2 \phi^\dagger \wedge \bm{*}(\delta \phi) \right\}
\\ &=&
 \frac{1}{2} \int \delta \phi^\dagger \wedge \left[ -\bm d  \bm{*}\bm d \, \phi  + m^2 \bm{*}(\phi) \right] + \frac{1}{2} \int \left[ -\bm d  \bm{*}\bm d \, \phi^\dagger  + m^2 \bm{*}(\phi^\dagger) \right] \wedge \delta \phi~, \nonumber
\end{eqnarray}
imposing the minimization of the action functional we end up with the following
equations of motion
\begin{equation}
\delta S = 0 ~~Ê\Rightarrow ~~ \bm{*}  \bm d  \bm{*}\bm d \, \phi  - m^2 \, \phi = 0~, ~~ \bm{*} \bm d  \bm{*}\bm d \, \phi^\dagger  -  m^2 \, \phi^\dagger  = 0~.
\end{equation}
We easily see that the map $ \bm{*}  \bm d  \bm{*}\bm d $ is identical to the action of the mass casimir $\square_\kappa$ on scalar fields
\begin{eqnarray}
&&\bm{*}  \bm d  \bm{*}\bm d \, \phi  =
\bm{*}  \bm d \, [\xi_a \tr \phi  \,\bm{*}( \bm e^a )]=
\frac{1}{5!} \eta^{ab} \varepsilon_{bcdef} \,\bm{*}  \bm d \, [ \xi_a \tr \phi  \, \bm e^c \wedge \bm e^d \wedge \bm e^e \wedge \bm e^f] = \nonumber \\
&& \frac{1}{5!} \eta^{ab} \varepsilon_{bcdef} \, \xi_g \xi_a \tr \phi \,  \bm{*} ( \bm e^c \wedge \bm e^d \wedge \bm e^e \wedge \bm e^f \wedge \bm e^g)= \\
&& \eta^{ab} \varepsilon_{bcdef} \varepsilon^{cdefg} \, \xi_g \xi_a \tr \phi \,  \bm{*} ( \bm{vol}^5) =  \eta^{ab} \, \xi_b \xi_a \tr \phi = \square_\kappa \tr \phi ~. \nonumber
\end{eqnarray}

\subsection{Noether theorem and energy-Momentum tensor}

The following current vector-valued four-form:
\begin{equation}
\bm j_a = \frac{1}{2} \left\{ (\chi_a \tr \phi^\dagger) \wedge \bm *\, \bm d \, \phi
+ \bm  * \, \bm d ({\sigma^b}_a \tr \phi^\dagger) \wedge (\chi_b \tr \phi) \right\} - \bm i_a (\mathscr{L}) ~,
\end{equation} 
where $\mathscr{L} = \mc L \, \bm{vol}^5 =  \frac{1}{2}  \left\{ - \phi^\dagger (\square_\kappa \tr  \phi )
+ m^2 \phi^\dagger \, \phi \right\} \bm{vol}^5$, is conserved on-shell, in the sense that it is a closed form when $\phi$ and $\phi^\dagger$ minimize the action. Let's prove it:
\begin{eqnarray}
\bm d \, \bm j_a = \frac{1}{2} && \left\{  \bm d (\chi_a \tr \phi^\dagger) \wedge \bm *\, \bm d \, \phi
+ \bm d \bm  * \, \bm d ({\sigma^b}_a \tr \phi^\dagger) \wedge (\chi_b \tr \phi) 
\right. \nonumber \\
&&\left.+  (\chi_a \tr \phi^\dagger) \wedge \bm d \bm *\, \bm d \, \phi
+ \bm  * \, \bm d ({\sigma^b}_a \tr \phi^\dagger) \wedge \bm d  (\chi_b \tr \phi) 
\right\} - \bm \pounds_{\chi_a}  (\mathscr{L}) ~,
\end{eqnarray}
using the equations of motion
$$
\bm d \, \bm  * \, \bm d \, (\phi) = m^2 \bm * \,(\phi) ~, \qquad
\bm d \, \bm  * \, \bm d \, (\phi^\dagger) = m^2 \bm * \,(\phi^\dagger) ~,
$$
\begin{eqnarray}
\bm d \, \bm j_a = \frac{1}{2} && \left\{  \bm d (\chi_a \tr \phi^\dagger) \wedge \bm *\, \bm d \, \phi
+ m^2 \, \bm * ({\sigma^b}_a \tr \phi^\dagger) \wedge (\chi_b \tr \phi) 
\right. \nonumber \\
&&\left.+ m^2 (\chi_a \tr \phi^\dagger) \wedge \bm * (\phi)
+ \bm  * \, \bm d ({\sigma^b}_a \tr \phi^\dagger) \wedge \bm d  (\chi_b \tr \phi) 
\right\} - \bm \pounds_{\chi_a}  (\mathscr{L}) ~,\end{eqnarray}
it's trivial to prove the identities
\begin{eqnarray*}
\bm  * \, \bm d ({\sigma^b}_a \tr \phi^\dagger) \wedge \bm d  (\chi_b \tr \phi) 
=\bm d ({\sigma^b}_a \tr \phi^\dagger) \wedge  \bm  * \, \bm d  (\chi_b \tr \phi)  ~,
\\
m^2 \, \bm * ({\sigma^b}_a \tr \phi^\dagger) \wedge (\chi_b \tr \phi) =
m^2 ({\sigma^b}_a \tr \phi^\dagger) \wedge \bm * (\chi_b \tr \phi) ~,
\end{eqnarray*}
then the first term becomes equal to the action of the Lie derivative over
the Lagrangian
\begin{equation}
\bm d \, \bm j_a = \frac{1}{2} \bm \pounds_{\chi_a}  \left\{ 
\bm d \,(\phi^\dagger) \wedge \bm *\, \bm d \, ( \phi)
+ m^2 \phi^\dagger \wedge \bm *\, (\phi)
\right\} - \bm \pounds_{\chi_a}  (\mathscr{L}) = 0~.
\end{equation}

The components of the current form are the components of the energy-momentum tensor:
\begin{equation}
\bm j_a =  \bm *(\bm e^b) \, T_{ab}~,
\end{equation}
these components are
\begin{equation}
T_{ab} = \frac{1}{2} \left\{ \left({\sigma^b}_b\chi _a\triangleright \phi ^{\dagger }\right)\left(\chi _c\triangleright \phi \right)+\left({\sigma ^c}_a \chi _b\triangleright \phi ^{\dagger }\right)\left(\chi _c \triangleright \phi \right) \right\} - \eta_{ab} \mc L ~,
\end{equation}
if we take a solution of the equations of motion (the order of the $x$\,s is relevant),
\begin{equation}
\phi = e^{i \vec k \cdot \vec x} e^{-i k_0 x_0} ~,  ~~  \eta^{ab}\chi_a(k) \chi_b(k) = m^2 ~,
\end{equation}
which is at the same time an eigenfunction of the $\xi_a$ vector fields:
\begin{equation}
\chi_a \tr \phi =  \chi_a(k) \phi ~, ~~ \chi_a \tr \phi^\dagger = \xi_b(k) \phi^\dagger ~,
\end{equation}
where $\chi_a(k) =  \left\{- \kappa \, \sinh \frac{k_0}{\kappa} - \frac{1}{2 \kappa} e^{k_0 / \kappa} |\vec k|^2   ,  e^{\lambda k_0} \vec k , - \kappa \,  \cosh \frac{k_0}{\kappa} - \frac{1}{2 \kappa} e^{k_0 / \kappa} |\vec k|^2    + \kappa \right\}$ and similarly for $\xi_a(k)$,
and evaluate the energy-momentum tensor over this solution we get
\begin{equation}
T_{ab} = - \frac{1}{2} \left[ \xi_a(k) \xi_b(k) + \xi_b(k) \xi_a(k) \right] \phi^\dagger \phi~.
\end{equation}

We are also able to identify a current 3-form associated to the symmetry under global phase transformations $\phi' = e^{i \alpha} \phi$,
\begin{equation}
\bm j = \bm*\left(\phi ^{\dagger } \, \bm d \phi  - \bm d \phi ^{\dagger }\,  \phi \right)~,
\end{equation}
which is conserved on-shell
\begin{equation}
\bm d \, \bm j = \bm d \phi ^{\dagger } \, \bm*\left( \bm d \phi \right) - \bm*\left( \bm d \phi ^{\dagger }\right) \, \bm d \phi +\phi ^{\dagger } \left( \bm d \bm* \bm d \, \phi \right) - \left( \bm d \bm* \bm d \, \phi ^{\dagger }\right) \phi = 0 ~.
\end{equation}

\section{Conclusions}
\label{Conclusions}

I defined constructively most of the structures that are needed to
do differential geometry on $\kappa$-Minkowski. These structures
are all covariant under the symmetries of this noncommutative
spacetime. They allow two kind of future developments: one is 
the study of field theory over $\kappa$-Minkowski, for which now
we are equipped with all of the necessary tools to construct
covariant field theories. We can define vector and tensor fields
with the differential forms of various degrees, we can multiply
these forms thanks to the Hodge-$*$ and the integral, which
allow to associate scalar numbers to every field, as required
by an action principle. We can act with a vector field on forms
through the Lie derivative and reduce their degree with the
internal derivative, and this is sufficiently powerful to construct
the conserved currents associated to the symmetries of the
spacetime. We are now armed with sufficient tools to start
doing serious quantum field theory on $\kappa$-Minkowski,
exploiting its symmetries in the correct way.
Another strand of studies which can take this paper as starting point
is a more geometrical study of the properties of $\kappa$-Minkowski,
which exploits the differential-form structures I defined to build
actual differential-geometric objects on it, like Cartan's frame
fields, a connection, torsion and curvature, and so on. This,
in the long term, might even lead to a proposal for a perturbative
construction of quantum gravity as a noncommutative field
theory, if $\kappa$-Minkowski proves to be a good fundamental
state for quantum gravity.

An interesting recent development in quantum gravity is the so-called
``Relative Locality'' proposal, which is a framework for interpreting
the classical remnants of quantum gravity effects in terms of a curved
momentum space \cite{AFKSfqxi,FlavioPRL,AFKS}.
I have shown in a paper \cite{kappaLocality} with G. Gubitosi that 
the $\kappa$-Poincar\'e quantum group fits perfectly in the framework
of this theory, and all the Hopf-algebra structures of this quantum
group are necessary to identify a coherent Relative Locality model.
This suggest the following interpretative scheme:
the ``classical'' field theory considered in this paper should only be 
understood as preparatory to a quantum field theory expressed in
terms of a path integral. Its classical limit $\hbar \to 0$ should eliminate
the noncommutativity of $\mc A$, but should leave a trace into
the symplectic structure of the phase space of particles.
Moreover, this phase space in the classical limit should tend to the
cotangent bundle of a curved momentum space, which I described
in \cite{kappaLocality}.

\subsection*{Acknowledgments}

I whish to thank Andrzej Sitarz, Shahn Majid, Jos\'e Gracia-Bond\'ia and Matteo Giordano for useful observations and comments.

\providecommand{\href}[2]{#2}\begingroup\raggedright\endgroup

\end{document}